\documentclass[11pt, letterpaper, reqno]{amsart}
\usepackage[utf8]{inputenc} 	
\usepackage{microtype} 			
\usepackage{geometry} 			
\usepackage{amsmath}			
\usepackage{amsthm}		 		
\usepackage{amssymb}	 		
\usepackage{esint}	 			
\usepackage{bm}					
\usepackage{mathrsfs}			
\usepackage{thmtools}
\usepackage{enumitem} 			
\usepackage{lmodern}			
\usepackage[T1]{fontenc}		
\usepackage[colorlinks=true, citecolor=cyan, urlcolor=magenta]{hyperref}
\usepackage{cleveref}
\usepackage[backend=biber, style=alphabetic, backref=true, firstinits=true, isbn=false, date=year, maxbibnames=6]{biblatex}
\makeatletter
\renewbibmacro{in:}{}
\DeclareFieldFormat{pages}{#1}
\renewcommand*{\bibnamedash}{%
	\leavevmode\raise +0.6ex\hbox to 5.5ex{\hrulefill}.\space\space}

\InitializeBibliographyStyle{\global\undef\bbx@lasthash}

\newbibmacro*{bbx:savehash}{\savefield{fullhash}{\bbx@lasthash}}

\renewbibmacro*{author}{%
	\ifboolexpr{
		test \ifuseauthor
		and
		not test {\ifnameundef{author}}
	}
	{%
		\iffieldequals{fullhash}{\bbx@lasthash}
		{\bibnamedash\addcomma\space}
		{\printnames{author}}%
		\usebibmacro{bbx:savehash}%
		\iffieldundef{authortype}
		{}
		{%
			\setunit{\addcomma\space}%
			\usebibmacro{authorstrg}%
		}%
	}
	{\global\undef\bbx@lasthash}%
}
\makeatother
\newtheorem{propositionx}{Proposition}[section]
\newenvironment{proposition}
{\pushQED{\qed}\propositionx}
{\popQED\endpropositionx}
\newenvironment{propositionp}
{\pushQED{\qed}\propositionx}
{\popQED\endpropositionx}
\newtheorem{theoremx}[propositionx]{Theorem}
\newenvironment{theorem}
{\pushQED{\qed}\theoremx}
{\popQED\endtheoremx}

\newtheorem{lemmax}[propositionx]{Lemma}
\newenvironment{lemma}
{\pushQED{\qed}\lemmax}
{\popQED\endlemmax}

\theoremstyle{remark}
\newtheorem*{remarkx}{Remark}
\newenvironment{remark}
{\pushQED{\qed}\remarkx}
{\popQED\endremarkx}
\newtheorem*{examplex}{Example}
\newenvironment{example}
{\pushQED{\qed}\examplex}
{\popQED\endexamplex}

\newcommand{\bbB}{\mathbb{B}}
\newcommand{\bbC}{\mathbb{C}}

\newcommand{\bbK}{\mathbb{K}}

\newcommand{\bbN}{\mathbb{N}}

\newcommand{\bbP}{\mathbb{P}}
\newcommand{\bbQ}{\mathbb{Q}}
\newcommand{\bbR}{\mathbb{R}}

\newcommand{\bbZ}{\mathbb{Z}}

\newcommand{\calB}{\mathcal{B}}

\newcommand{\calE}{\mathcal{E}}
\newcommand{\calF}{\mathcal{F}}

\newcommand{\calN}{\mathcal{N}}

\newcommand{\calP}{\mathcal{P}}

\newcommand{\calS}{\mathcal{S}}
\newcommand{\calT}{\mathcal{T}}

\newcommand{\calX}{\mathcal{X}}

\newcommand{\scrD}{\mathscr{D}}

\newcommand{\scrF}{\mathscr{F}}

\newcommand{\scrX}{\mathscr{X}}
\newcommand{\scrY}{\mathscr{Y}}

\newcommand{\frakN}{\mathfrak{N}}

\newcommand{\dd}{\,\mathrm{d}}

\title{A strengthened Orlicz--Pettis theorem via It\^o--Nisio}
\author{Ethan Sussman}
\date{October 20th, 2022 (Last update). July 1st, 2021 (Preprint).}
\email{ethanws@mit.edu}
\address{Department of Mathematics, Massachusetts Institute of Technology, Massachusetts 02139-4307, USA}

\begin{document}
	
\begin{abstract}
	In this note we deduce a strengthening of the Orlicz--Pettis theorem from the It\^o--Nisio theorem. The argument shows that given any series in a Banach space which isn't summable (or more generally unconditionally summable), we can \emph{construct} a (coarse-grained) subseries with the property that -- under some appropriate notion of ``almost all'' -- almost all further subseries thereof fail to be weakly summable. Moreover, a strengthening of the It\^o--Nisio theorem by Hoffmann-J{\o}rgensen allows us to replace `weakly summable' with `$\tau$-weakly summable' for appropriate topologies $\tau$ weaker than the weak topology. A treatment of the It\^o--Nisio theorem for admissible $\tau$ is given. 
\end{abstract}

\maketitle

\tableofcontents

\section{Introduction}

Let $\scrX$ denote a Banach space over $\bbK \in \{\bbR,\bbC\}$.
Call a subset $\tau \subseteq2^{\scrX}$ an \emph{admissible topology} on $\scrX$ if 
\begin{enumerate}
	\item it is an LCTVS\footnote{By `LCTVS' we mean a \emph{Hausdorff} locally convex topological vector space, so we follow the conventions in \cite{Rudin}.}-topology on $\scrX$ identical to or weaker than the norm (a.k.a.\ strong) topology under which the norm-closed unit ball $\bbB =\{x\in \scrX: \lVert x \rVert\leq 1\}$ is $\tau$-closed, and
	\label{it:admissible_main}
	\item if $\scrX$ is \emph{not} separable, then $\tau$ is at least as strong as the weak topology.
	\label{it:admissible_two}
\end{enumerate}
Cf.\ \cite{Hoffmann1974}, from which the separable case of this definition arises.  By the Hahn-Banach separation theorem, if $\tau$ is an admissible topology then the $\tau$-weak topology (a.k.a.\ $\sigma(\scrX,\scrX_\tau^*)$-topology) is also admissible (see \Cref{lem:weak}). 

Besides the norm topology itself, which is trivially admissible (and uninteresting below), the most familiar example of an admissible topology on $\scrX$ is the weak topology.
Many others arise in functional analysis. For example, given a compact Riemannian manifold $M$, for most function spaces $\scrF$ it is the case that the $\sigma(\scrF,C^\infty(M))$-topology (a.k.a.\ the topology of distributional convergence) is admissible. An even weaker typically admissible topology is that on $\scrF$ generated by the functionals $\langle - ,\varphi_n \rangle:\scrD'(M)\to \bbC$ for $\varphi_0,\varphi_1,\varphi_2,\cdots$ the eigenfunctions of the Laplacian. 

Denote by $\smash{\scrX^{\bbN}}$ the vector space of all $\scrX$-valued sequences $\{x_n\}_{n=0}^\infty \subseteq\scrX$. In the usual way, we identify such sequences with $\scrX$-valued formal series (and denote accordingly). 
We say that a formal series $\sum_{n=0}^\infty x_n \in \scrX^\bbN$ is ``$\tau$-summable'' if $\smash{\sum_{n=0}^N }x_n \in \scrX$ converges as $N\to\infty$ in $\scrX_\tau$. 

Consider the following (slightly generalized) version of the  Orlicz--Pettis theorem  \cite{Orlicz1929}:
\begin{theorem} \label{thm:OP}
	Suppose that $\tau$ is an admissible topology on $\scrX$.
	 If $\sum_{n=0}^\infty x_n \in \scrX^\bbN$ fails to be unconditionally summable in the norm topology, then 
	 \begin{itemize}
	 	\item there exist some $\epsilon_0,\epsilon_1,\epsilon_2,\cdots \in \{-1,+1\}$
	 	such that the sequence	$\Sigma(\{\epsilon_n\}_{n=0}^\infty)=\{\Sigma_N\}_{N=0}^\infty$ defined by 
	 	\begin{equation}
	 		\Sigma_N = \sum_{n=0}^N \epsilon_n x_n
	 		\label{eq:Sigma}
	 	\end{equation}
 		does not $\tau$-converge as $N\to\infty$ to any element of $\scrX$, and
 		\item there exist some $\chi_0,\chi_1,\chi_2,\cdots \in \{0,1\}$ 
 		such that the sequence	$S(\{\epsilon_n\}_{n=0}^\infty)=\{S_N\}_{N=0}^\infty$ defined by 
 		\begin{equation}
 			S_N = \sum_{n=0}^N \chi_n x_n, 
 			\label{eq:S}
 		\end{equation}
 		does not $\tau$-converge as $N\to\infty$ to any element of $\scrX$. 
	 \end{itemize}
 	In particular, this applies if $\sum_{n=0}^\infty x_n$ is not summable in the norm topology.
\end{theorem} 
\begin{remark}	\label{rem:eq}
	From the formulas 
	\begin{align}
		\Sigma_N(\{\epsilon_n\}_{n=0}^N) &= S_N(\{2^{-1}(1+\epsilon_n)\}_{n=0}^N) -S_N(\{2^{-1}(1-\epsilon_n)\}_{n=0}^N) 
		\label{eq:misc_j21}\\
		S_N(\{\chi_n\}_{n=0}^N) &= 2^{-1} \Sigma_N(\{ 1\}_{n=0}^N) + 2^{-1} \Sigma_N(\{ 2\chi_n-1\}_{n=0}^N), 
		\label{eq:misc_j22}
	\end{align}
	we deduce that 
	$\Sigma(\{\epsilon_n\}_{n=0}^\infty)$ is $\tau$-convergent for all $\{\epsilon_n\}_{n=0}^\infty \in \{-1,+1\}^\bbN$ if and only if $S(\{\chi_n\}_{n=0}^\infty)$ is $\tau$-convergent for all $\{\chi_n\}_{n=0}^\infty \in \{0,1\}^\bbN$. We will phrase the discussion below in terms of whichever of $\Sigma(-),S(-)$ is convenient, but this equivalence should be kept in mind. 
	
	See \Cref{prop:eq} for the probabilistic version of this remark.
\end{remark}

\begin{example}
	Let $M$ be a compact Riemannian manifold and $\scrF\subseteq\scrD'(M)$ be a function space on $M$. Let $\tau$ denote the topology generated by the functionals $\langle - ,\varphi_n \rangle_{L^2(M)}$, where $\varphi_0,\varphi_1,\varphi_2,\cdots$ denote the eigenfunctions of the Laplace-Beltrami operator. 
	Suppose that $\tau$ is admissible. This holds, for example, if $\scrF$ is an $L^p$-based Sobolev space for $p\in [1,\infty)$.
	
	Then, for any $\{x_n\}_{n=0}^\infty \subseteq\scrF$, the formal series $\sum_{n=0}^\infty x_n$ is unconditionally summable in $\scrF$ (in norm) if and only if 
	\begin{equation}
		\sum_{n=0}^\infty | \langle x_n ,\varphi_m \rangle | < \infty 
	\end{equation}
	for all $m\in \bbN$ and, for all $\{\chi_n\}_{n=0}^\infty \subseteq\{0,1\}$, there exists an element $S(\{\chi_n\}_{n=0}^\infty)\in \scrF$ whose $m$th Fourier coefficient is given by 
	\begin{equation}
		\langle S(\{\chi_n\}_{n=0}^\infty), \varphi_m\rangle=\sum_{n=0}^\infty \chi_n \langle x_n ,\varphi_m \rangle.
	\end{equation} 
\end{example}

We focus on Banach spaces -- as opposed to more general LCTVSs -- for simplicity. Most of the considerations below apply equally well to Fr\'echet spaces. There is a long history of variants of the Orlicz--Pettis theorem for various sorts of TVSs \cite{Dierolf}. A short proof of the Orlicz--Pettis theorem for Banach spaces can be found in \cite{Bessaga1958}, and a textbook presentation can be found in \cite{Meg}. The proof below has much in common with a probabilistic proof \cite{DiestelMain} based on the Bochner integral (due to Kwapie\'n). 

The proof below is nonconstructive, in the following sense: upon being given a formal series $\sum_{n=0}^\infty x_n \in \scrX^\bbN$ which fails to be unconditionally summable, we do not construct any particular sequence $\{\epsilon_n\}_{n=0}^\infty \subseteq\{-1,+1\}$ such that $\Sigma(\{\epsilon_n\}_{n=0}^\infty) \subseteq\scrX$ fails to converge in $\scrX_\tau$, or any particular $\{\chi_n\}_{n=0}^\infty \subseteq\{0,1\}$ such that $S(\{\chi_n\}_{n=0}^\infty)\subseteq\scrX$ fails to converge in $\scrX_\tau$. 
All proofs of the Orlicz--Pettis theorem seem to be nonconstructive in this regard.
We do, however, construct a function 
\begin{equation} 
	\calE : \{ \{x_n\}_{n=0}^\infty \in \scrX^\bbN \text{ not unconditionally summable}  \} \to 2^{\{-1,+1\}^\bbN },
	\label{eq:coarse_graining}
\end{equation} 
such that, when $\{x_n\}_{n=0}^\infty$ is not unconditionally summable, $\Sigma(\{\epsilon_n\}_{n=0}^\infty)$ and $S(\{2^{-1}(1-\epsilon_n)\}_{n=0}^\infty)$ both fail to be $\tau$-summable for $\bbP_{\mathrm{Coarse}}$-almost all sequences $\{\epsilon_n\}_{n=0}^\infty \in \calE$, where
\begin{equation}
	\bbP_{\mathrm{Coarse}} : \operatorname{Borel}(\{-1,+1\}^\bbN)|_{\calE(\{x_n\}_{n=0}^\infty)} \to [0,1] 
	\label{eq:sampling-procedure}
\end{equation}
is a probability measure on the subspace $\sigma$-algebra 
\begin{equation} 
	\operatorname{Borel}(\{-1,+1\}^\bbN)|_{\calE(\{x_n\}_{n=0}^\infty)} = \{S\cap\calE(\{x_n\}_{n=0}^\infty) : S\in  \operatorname{Borel}(\{-1,+1\}^\bbN)\}.
\end{equation}
So, while the proof is nonconstructive, it is only just. 
Put more colorfully, the proof follows the ``hay in a haystack'' philosophy familiar from applications of the probabilistic method to combinatorics \cite{AlonSpencer}: using an appropriate sampling procedure, we choose a random subseries and show that -- with ``high probability'' (which in this case means probability one) -- it  has the desired property. 

Precisely, letting $\bbP_{\mathrm{Haar}}$ denote the Haar measure on the Cantor group $\{-1,+1\}^{\bbN}\cong \bbZ_2^\bbN$ \cite{DiestelMain} (which is a compact topological group under the product topology, by Tychonoff's theorem):
\begin{theorem}[Probabilist's Orlicz--Pettis Theorem] \label{thm:PO}
	Suppose that $f:\bbN\to \bbN$ is a function such that $|f^{-1}(\{n\})|<\infty$ for all $n\in \bbN$. If $\calT\subseteq\bbN$ is infinite and satisfies
	\begin{equation}
		\limsup_{n\to\infty, n\in \calT} \Big\rVert \sum_{ n_0 \in f^{-1}(\{n\}) } x_{n_0} \Big\rVert >0, 
		\label{eq:misc_g41}
	\end{equation}
	then it is the case that, for $\bbP_{\mathrm{Haar}}$-almost all $\{\epsilon_n\}_{n=0}^\infty \in \{-1,+1\}^\bbN$, the formal series 
	\begin{equation}
		\sum_{n=0,f(n)\in \calT}^\infty \epsilon_{f(n)} x_n \in \scrX^\bbN, \qquad \sum_{n=0,f(n)\in \calT}^\infty \frac{1}{2}(1-\epsilon_{f(n)}) x_n  \in \scrX^\bbN
		\label{eq:misc_h31}
	\end{equation}
	both fail to be $\tau$-summable.
\end{theorem}

The relation to Orlicz--Pettis is as follows. If $\sum_{n=0}^\infty x_n \in \scrX^\bbN$ is not unconditionally summable, then we can find some pairwise disjoint, finite subsets $\calN_0,\calN_1,\calN_2,\cdots  \subseteq\bbN$ such that 
\begin{equation}
	\inf_{N\in \bbN} \Big\lVert \sum_{n\in \calN_N} x_n \Big\rVert  > 0.
\end{equation}
We can then choose some $f:\bbN\to\bbN$ such that $f(n)=f(m)$ if and only if either $n=m$ or $n,m\in \calN_N$ for some $N\in \bbN$. Thus, if we set $\calT=\bbN$, \cref{eq:misc_g41} holds. Appealing to \Cref{thm:PO}, we conclude that, for $\bbP_{\mathrm{Haar}}$-almost all $\{\epsilon_n\}_{n=0}^\infty$, 
the formal series 
\begin{equation}
	\sum_{n=0}^\infty \epsilon_{f(n)} x_n \in \scrX^\bbN, \qquad \sum_{n=0}^\infty \frac{1}{2}(1-\epsilon_{f(n)}) x_n  \in \scrX^\bbN
\end{equation}
both fail to be $\tau$-summable.
\Cref{thm:OP}, therefore, follows from \Cref{thm:PO}. 
The connection with \cref{eq:coarse_graining}, \cref{eq:sampling-procedure} is that we can choose $f$ such that $\calE$ is the set of $\{\epsilon_n\}_{n=0}^\infty \in \{-1,+1\}^\bbN$ such that $\epsilon_n=\epsilon_m$ whenever $f(n)=f(m)$, and $\bbP_{\mathrm{Coarse}}$ is $\bbP_{\mathrm{Haar}}$ conditioned on the event that $\{\epsilon_n\}_{n=0}^\infty \in \calE$.

\begin{remark}
	The Haar measure on the Cantor group is the unique measure on $\operatorname{Borel}(\{-1,+1\}^\bbN) = \sigma(\{\epsilon_n\}_{n=0}^\infty)$ such that if we define $\epsilon_n: \{-1,+1\}^\bbN \to \{-1,+1\}$ by $\epsilon_n:\{\epsilon'_m\}_{m=0}^\infty \mapsto \epsilon'_n$, the random variables
	$\epsilon_0,\epsilon_1,\epsilon_2,\cdots$
	are i.i.d.\ Rademacher random variables.
\end{remark}
\begin{remark}
	It suffices to prove the theorems above when $\scrX$ is separable. Indeed, if $\scrX$ is not separable and $\scrY$ denotes the norm-closure of the span of $x_0,x_1,x_2,\cdots \in \scrX$, then, for any $\{\lambda_n\}_{n=0}^\infty \subseteq\bbK$, 
	\begin{equation}
		\tau\!-\!\!\lim_{N\to\infty} \sum_{n=0}^N \lambda_n x_n  
		\label{eq:misc_ttt}
	\end{equation}
	exists in $\scrX$ if and only if it exists in $\scrY$. (This is a consequence of the requirement that $\tau$ be at least as strong as the weak topology, so the limit in \cref{eq:misc_ttt} is also a weak limit. Norm-closed convex subsets of $\scrX$ are weakly closed by Hahn-Banach, so this implies that $\scrY$ is $\tau$-closed.)
	
	The subspace topology on $\scrY\hookrightarrow \scrX_\tau$ is admissible, and $\scrY$ is separable, so we can deduce \Cref{thm:OP} and \Cref{thm:PO} for $\scrX$ from the same theorems for $\scrY$. 
\end{remark}
\begin{remark}
	\label{rem:sep}
	If $\scrX$ is not separable and $\tau$ not at least as strong as the weak topology, then the conclusions of these theorems may fail to hold, even if the norm-closed balls in $\scrX$ are $\tau$-closed. 
	As a simple counterexample, let $\scrX=L^\infty [0,1]$, and let $\tau$ be the $\sigma(L^\infty,L^1)$-topology. 
	This being a weak-$*$ topology, the norm-closed balls are $\tau$-closed (and even $\tau$-compact).
	Let 
	\begin{equation}
		\Sigma_N(t) = t^N,
	\end{equation}
	$x_n(t) = \Sigma_{n}(t) -\Sigma_{n-1}(t)$ for $n\geq 1$, $x_0(t) = \Sigma_0(t)$. Then, the series $\sum_{n=0}^\infty x_n$ is $\tau$-subseries summable, being $\tau$-summable to the identically zero function. But, $\Sigma_N$ does not converge to zero uniformly, so $\sum_{n=0}^\infty x_n$ is not strongly summable.
\end{remark}
\begin{remark}	\label{rem:count}
	When $\scrX$ is separable, it suffices to consider the case when $\tau$ is the topology generated by a countable norming set of functionals. Recall that a subset $\calS\subseteq\scrX_\tau^*$ is called \emph{norming} if 
	\begin{equation}
		\lVert x \rVert = \sup_{\Lambda \in \calS} |\Lambda x |
	\end{equation}
	for all $x\in \scrX$. We can scale the members of a norming subset to get another norming subset whose members $\Lambda$ satisfy $\smash{\lVert \Lambda \rVert_{\scrX^*}}=1$, and this generates the same topology. If $\tau$ is admissible, then (by the Hahn-Banach theorem and separability) there exists a countable norming subset $\calS \subseteq\smash{\scrX_\tau^*}$ (see \Cref{lem:countable_norming}). 
	Whenever $\calS\subseteq\scrX^*_\tau$ is a countable norming subset, the $\sigma(\scrX,\calS)$-topology is admissible as well (see \Cref{lem:S_admissible}), and identical with or weaker than $\tau$. 
\end{remark}

It is not necessary to consider probability spaces other than 
\begin{equation}
	(\{-1,+1\}^\bbN,\operatorname{Borel}(\{-1,+1\}^\bbN), \bbP_{\mathrm{Haar}}),
\end{equation} 
but it will be convenient to have a bit more freedom.
Let $(\Omega,\calF,\bbP)$ denote a probability space on which i.i.d.\ Bernoulli random variables 
\begin{equation} 
	\chi_0,\chi_1,\chi_2,\cdots : \Omega\to \{0,1\} 
\end{equation} 
are defined. For example, 
\begin{equation} 
	(\Omega,\calF,\bbP) = (\{-1,+1\}^\bbN,\operatorname{Borel}(\{-1,+1\}^{\bbN}),\bbP_{\mathrm{Haar}}),
\end{equation}  
in which case we set $\chi_n=(1/2)(1-\epsilon_n)$.
Given this setup and given a formal series $\sum_{n=0}^\infty x_n \in \scrX^\bbN$, we can construct a random formal subseries $S: \Omega\to \scrX^\bbN$ by
\begin{equation}
	S(\omega) = \sum_{n=0}^\infty \chi_n(\omega) x_n. 
	\label{eq:Si}
\end{equation}
This is a measurable function from $\Omega$ to $\scrX^{\bbN}$ when $\scrX$ is separable (see \Cref{lem:meas_0})

Suppose that $\scrX$ is separable. 
Given any Borel subset $\mathtt{P} \subseteq\scrX^\bbN$ the probability $\mathbb{P}(S^{-1}(\mathtt{P})) \in [0,1]$ of the ``event'' $S \in \mathtt{P}$ is well-defined. 
Given some ``property'' $\mathtt{P}$ -- which we identify with a not-necessarily-Borel subset $\mathtt{P}\subseteq \scrX^\bbN$ -- that a formal series may or may not possess, to say that almost all subseries of $\sum_{n=0}^\infty x_n$ have property $\mathtt{P}$ 
means that there exists some $F\in \calF$ with 
\begin{equation} 
	\bbP(F) = 1
	\label{eq:fm}
\end{equation} 
and $\omega \in F \Rightarrow S(\omega) \in \mathtt{P}$. 
In this case, we say that $S$ has the property $\mathtt{P}$ for $\bbP$-almost all $\omega$. (Note that we do not require $S^{-1}(\mathtt{P}) \in \calF$, although this is automatic if $\mathtt{P}$ is Borel, and can be arranged by passing to the completion of $\mathbb{P}$.)
Analogous locutions will be used for random formal series generally.
If $\mathtt{P}$ is Borel then $S(\omega)$ will have the property $\mathtt{P}$ for $\bbP$-almost all $\omega \in \Omega$ if and only if $\bbP(S^{-1}(\mathtt{P}))=1$.

In order to prove the theorems above, we use the following variant of a theorem of It{\^o} and Nisio \cite{ItoNisio} refined by Hoffmann-J{\o}rgensen \cite{Hoffmann1974}:
\begin{theorem} \label{thm:IN} 
	Suppose that $\tau$ is an admissible topology on $\scrX$. 
	Let 
	\begin{equation}
		\gamma_{0},\gamma_1,\gamma_2,\cdots : \Omega\to \{-1,+1\}
	\end{equation}
	be independent, symmetric random variables on $(\Omega,\calF,\bbP)$. If $\scrX$ is a Banach space and $\{x_n\}_{n=0}^\infty \in \scrX^\bbN$, the following are equivalent:
	\begin{enumerate}[label=(\Roman*)]
		\item for $\bbP$-almost all $\omega\in \Omega$, $\sum_{n=0}^\infty \gamma_n(\omega) x_n$ is summable in $\scrX$, 
		\label{it:IN1}
		\item for $\bbP$-almost all $\omega\in \Omega$, $\sum_{n=0}^\infty \gamma_n(\omega) x_n$ is $\tau$-summable, i.e.\ summable in $\scrX_\tau$.
		\label{it:IN2}
	\end{enumerate}
	Moreover, whether or not the conditions above hold depends only on $\{x_n\}_{n=0}^\infty$ and the laws of each of $\gamma_0,\gamma_1,\gamma_2,\cdots$.  
\end{theorem}

This result is essentially contained in \cite{Hoffmann1974}, but, since our formulation is slightly different, we present a proof in \S\ref{sec:IN} below.

See \cite{ABS} for a modern account of the It\^o--Nisio  result in the case when $\tau$ is the weak topology. Our proof follows theirs.

A special case of this theorem was stated in \cite{me}, and the proof was sketched. This paper fills in some details of that sketch.\footnote{See \cite[Thm. 3.11]{me}. The statement there involves convergence in probability, but the proof in \S\ref{sec:IN} below applies.}

\begin{remark}	
	We will refer to \Cref{thm:IN} as ``the It\^o--Nisio theorem,'' with the following three caveats:
	\begin{itemize}
		\item Unlike in the usual It\^o--Nisio theorem, we do not discuss convergence in probability.
		\label{it:cv1}
		\item 
		The result is often stated with general Bochner-measurable symmetric and independent random variables $x_n(\omega):\Omega\to \scrX^\bbN$ in place of $\gamma_n(\omega)x_n$. (A $\scrX$-valued random variable $X$ will be called \emph{symmetric} if $X$ and $-X$ are equidistributed, i.e.\ have the same law.\footnote{Note that, if $\bbK=\bbC$, this convention differs from some in the literature, in particular \cite[Definition 6.1.4]{ABS}. (We use `symmetric' when they would use `real-symmetric.')}) In fact, \Cref{thm:IN} implies the more general version via a rerandomization argument.
		\item It\^o and Nisio only consider the case when $\tau$ is the weak topology, the generalization to admissible $\tau$ being the result of \cite{Hoffmann1974}. 
	\end{itemize}
\end{remark}
\begin{remark} 
	A strengthening of the It\^o--Nisio result in the case when $\scrX$ does not admit an isometric embedding $c_0\hookrightarrow \scrX$  is essentially contained -- and explicitly conjectured -- in \cite{Hoffmann1974}. 
	The proof is due to 
	Kwapie\'n \cite{Kwapien1974}. If (and only if) $\scrX$ does not admit an isometric embedding $c_0\hookrightarrow \scrX$, then (I), (II) in \Cref{thm:IN} are equivalent to
	\begin{itemize}
		\item[(III)] for almost all $\omega\in \Omega$, $\sup_{N\in \bbN} \lVert \sum_{n=0}^N \epsilon_n(\omega) x_n \rVert < \infty$. 
	\end{itemize}
	(The event described above, that of ``uniform boundedness,'' is also measurable. See \Cref{lem:meas}.)
	
	Recall that -- by the uniform boundedness principle -- the weak convergence of a sequence $\{X_N\}_{N=0}^\infty \subseteq\scrX$ implies that $\sup_N \lVert X_N \rVert< \infty$, so (II) implies (III) when $\tau$ is the weak topology. Condition (I) obviously implies (III), so by the It\^o--Nisio theorem (once we've proven it), (II) implies (III) for any admissible $\tau$. 
	The converse obviously does not hold if $\scrX$ admits an isometric embedding $c_0\hookrightarrow \scrX$. 
\end{remark} 
\begin{remark} 	
	By \Cref{lem:meas}, the events described in (I), (III) above are measurable, and so, \Cref{thm:IN} is a statement about their probabilities. If $\scrX$ is separable and $\tau$ is the topology generated by a countable norming collection of functionals, the event in (II) is measurable as well. 
	It is a consequence of \Cref{thm:IN} that, if the probability space $(\Omega,\calF,\bbP)$ is complete, then (II) is measurable regardless. 
\end{remark}

An outline for the rest of this note is as follows:
\begin{itemize}
	\item In \S\ref{sec:preliminaries}, we fill in some measure-theoretic details related to the main line of argument. 
	\item We prove the It\^o--Nisio theorem in \S\ref{sec:IN} using a version of the standard argument based on uniform tightness and L\'evy's maximal inequality.
	\item Using \Cref{thm:IN}, we prove the probabilist's Orlicz--Pettis theorem in \S\ref{sec:main}
\end{itemize}

\section{Measurability}
\label{sec:preliminaries}

Let $\scrX$ be an arbitrary separable Banach space over $\bbK \in \{\bbR,\bbC\}$, and let $\tau$ be an admissible topology on it.
Below, $\epsilon_0,\epsilon_1,\epsilon_2,\cdots$ will be as in \Cref{thm:IN}, i.i.d.\ Rademacher random variables $\Omega\to \{-1,+1\}$. Similarly, $\chi_0,\chi_1,\chi_2,\cdots$ will be i.i.d.\ uniformly distributed $\Omega\to \{0,1\}$. 

\begin{lemma}
	The function $S:\Omega\to \scrX^\bbN$ defined by \cref{eq:Si} is measurable with respect to the Borel $\sigma$-algebra $\operatorname{Borel}(\scrX^\bbN)$, so it is a well-defined random formal $\scrX$-valued series.
	\label{lem:meas_0}
\end{lemma} 
\begin{proof}
	The Borel $\sigma$-algebra of a countable product of separable metric spaces agrees with the product $\calP$ of the Borel $\sigma$-algebras of the individual factors \cite[Lemma 1.2]{Kallenberg}.  So,   $\operatorname{Borel}(\scrX^\bbN)=\sigma(\operatorname{eval}_n:n\in \bbN) = \calP$, where 
	\begin{equation}
		\operatorname{eval}_n:\scrX^\bbN\to \scrX
	\end{equation}
	is shorthand for the map $\sum_{n=0}^\infty x_n \mapsto x_n$. 
	To deduce that $S$ is Borel measurable, we just observe that it is measurable with respect to the $\sigma$-algebra $\sigma(\operatorname{eval}_n: n\in \bbN)$, since $\mathrm{eval}_n\circ S(\omega) = \chi_n(\omega) x_n$.
\end{proof}  

Let $\mathtt{P}_{\mathrm{I}},\mathtt{P}_{\mathrm{II}},\mathtt{P}_{\mathrm{III}} \subseteq\scrX^\bbN$ denote the sets of (I) strongly summable formal series, (II) $\tau$-summable formal series, and (III) bounded formal series, 
respectively. 
In other words, 
\begin{equation}\textstyle{
\mathtt{P}_{\mathrm{I}} =\{\{x_n\}_{n=0}^\infty  \in \scrX^\bbN :\lim_{N\to\infty}\sum_{n=0}^N x_n \text{ exists in $\scrX$} \}},
\end{equation}
\begin{equation}\textstyle{
\mathtt{P}_{\mathrm{II}} =\{\{x_n\}_{n=0}^\infty \in \scrX^\bbN :\tau\!-\!\lim_{N\to\infty}\sum_{n=0}^N x_n \text{ exists in $\scrX_\tau$} \}},
\end{equation}
\begin{equation}\textstyle{
\mathtt{P}_{\mathrm{III}} =\{\{x_n\}_{n=0}^\infty  \in \scrX^\bbN :\sup_{N\in \bbN} \lVert \sum_{n=0}^N x_n \rVert<\infty\}}.
\end{equation}
Likewise, given a countable norming subset $\calS \subseteq\scrX^*_\tau $, let 
\begin{equation}\textstyle{
	\mathtt{P}_{\mathrm{II}'} = \mathtt{P}_{\mathrm{II}'}(\calS) = \{\{x_n\}_{n=0}^\infty \in \scrX^\bbN :\calS\!-\!\lim_{N\to\infty}\sum_{n=0}^N x_n \text{ exists in $\scrX_{ \sigma(\scrX,\calS) }$} \}}
\end{equation}
denote the set of  $\calS$-weakly summable formal $\scrX$-valued series. 

\begin{lemma}\label{lem:meas}
	$\mathtt{P}_{\mathrm{I}},\mathtt{P}_{\mathrm{II}'},\mathtt{P}_{\mathrm{III}} \in \operatorname{Borel}(\scrX^\bbN)$. 
	Consequently, given any random formal series $\Sigma : \Omega\to \scrX^\bbN$, $\Sigma^{-1}(\mathtt{P}_i) \in \calF$ for each $i \in \{\mathrm{I},\mathrm{II}',\mathrm{III}\}$. 
\end{lemma}	

\begin{proof}
	For each $M,N\in \bbN$, the function $\frakN_{N,M}:\scrX^\bbN\to \bbR$ given by 
	\begin{equation}
	\frakN_{N,M}(\{x_n\}_{n=0}^\infty) =  \Big\lVert \sum_{n=M}^N x_n \Big\rVert
	\end{equation}
	satisfies $\frakN_{N,M}^{-1}(S) \in \calP$ for all $S\in \operatorname{Borel}(\bbR)$. Therefore, $\mathtt{P}_{\mathrm{III}} = \cup_{R\in \bbN} \cap_{N\in \bbN} \frakN_{N,0}^{-1}([0,R])$ is in $\calP$, as is 
	\begin{equation} 
	\mathtt{P}_{\mathrm{I}} = \bigcap_{R\in \bbN^+} \bigcup_{M\in \bbN} \bigcap_{N\geq M} \frakN_{N,M}^{-1}([0,1/R]). 
	\end{equation} 
	
	Let $\scrX_0\subseteq\scrX$ denote a dense countable subset. 
	\emph{Claim}: a  sequence $\{X_N\}_{N=0}^\infty \subseteq\scrX$ converges $\calS$-weakly if and only if for each rational $\varepsilon>0$ there exists $X_{\approx} = X_{\approx}(\varepsilon) \in \scrX_0$ such that for each $\Lambda \in \calS$ there exists a $N_0=N_0(\varepsilon,\Lambda) \in \bbN$ such that 
	\begin{equation} 
	|\Lambda(X_N-X_{\approx})|< \varepsilon 
	\label{eq:LNe}
	\end{equation} 
	for all $N\geq N_0$.
	\begin{itemize}
		\item Proof of `only if:' if $X_N\to X$ $\calS$-weakly, then, for each $\varepsilon>0$, choose $X_{\approx} = X_{\approx}(\varepsilon) \in \scrX_0$ such that $\lVert X - X_{\approx} \rVert < \varepsilon/2$, and for each $\Lambda \in \calS$ choose $N_0(\varepsilon,\Lambda)$ such that $|\Lambda(X_N-X)|<\varepsilon/2$ for all $N\geq N_0$.  
		
		Since the elements of $\calS$ have operator norm at most one, $|\Lambda(X - X_{\approx}) | < \varepsilon/2$.
		
		Combining these two inequalities, \cref{eq:LNe} holds for all $N\geq N_0$.
		\item Proof of `if:' suppose we are given $X_{\approx}(\varepsilon)$ with the desired property. First, observe that $\{X_{\approx}(1/N)\}_{N=1}^\infty$ is Cauchy. Indeed, it follows from the definition of the $X_{\approx}(\varepsilon)$ that  $|\Lambda(X_{\approx}(\varepsilon)-X_{\approx}(\varepsilon'))| < \varepsilon+\varepsilon'$ for all $\Lambda \in \calS$, which implies (since $\calS$ is norming) that $\lVert X_{\approx}(\varepsilon)-X_{\approx}(\varepsilon') \rVert \leq \varepsilon+\varepsilon'$. 
		So, by the completeness of $\scrX$, there exists some $X\in \scrX$ such that
		\begin{equation}
		\lim_{N\to\infty} X_{\approx}(1/N) = X. 
		\end{equation}
		We now need to show that, as $N\to\infty$, $X_N\to X$ $\calS$-weakly. Indeed, given any $\Lambda \in \calS$ and $M\in \bbN^+$, 
		\begin{equation}
		|\Lambda(X_N-X)|\leq |\Lambda(X_N-X_{\approx}(1/M))| + |\Lambda(X-X_{\approx}(1/M))|.
		\end{equation}
		Given any $\varepsilon>0$, pick $M$ such that $1/M<\varepsilon/2$ and such that $\lVert X_{\approx}(1/M) - X \rVert < \varepsilon/2$. 
		Since the elements of $\calS$ have operator norm at most one, $|\Lambda(X-X_{\approx}(1/M))|<\varepsilon/2$. By the hypothesis of this direction, we can choose $N_0 = N_0(\varepsilon,\Lambda)$ sufficiently large such that $|\Lambda(X_N-X_{\approx}(1/M))| < 1/M<\varepsilon/2$ for all $N\geq N_0$. Therefore, $|\Lambda(X_N-X)| < \varepsilon$ for all $N\geq N_0$. 
		It follows that $X_N\to X$ $\calS$-weakly.
	\end{itemize}
	We therefore conclude that 
	\begin{equation} 
	\mathtt{P}_{\mathrm{II}'} =  \bigcap_{\varepsilon>0,\varepsilon \in \bbQ} \bigcup_{X_{\approx}\in \scrX_0} \bigcap_{\Lambda \in \calS} \bigcup_{M \in \bbN} \bigcap_{N\geq M} \{\{x_n\}_{n=0}^\infty : |\Lambda(X_N - X_{\approx})| <  \varepsilon\}
	\end{equation} 
	is in $\calP$ as well, where $X_N=x_0+\cdots+x_{N-1}$, which depends measurably on $\{x_n\}_{n=0}^\infty$.
	
\end{proof}

\begin{remark}
	We do not address the question of when $\mathtt{P}_{\mathrm{II}}$ is Borel.
	Even when $\scrX^*_\tau$ is not second countable, it can be the case that $\mathtt{P}_{\mathrm{II}} \in \calP$. For example, if $\scrX = \ell^1(\bbN)$, then sequential weak convergence is equivalent to sequential strong convergence \cite[Theorem 6.2]{Carothers}, and hence $\mathtt{P}_{\mathrm{I}} = \mathtt{P}_{\mathrm{II}}$.
\end{remark}

	Let $\pi_N:\scrX^\bbN\to \scrX^\bbN$ denote the left-shift map $\sum_{n=0}^\infty x_n \mapsto \sum_{n=0}^\infty x_{n+N}$. Let $\pi_N^* \calP = \{ \pi_N^{-1}(S):S\in \calP\}$. 
\begin{lemma} 	\label{lem:01}
	Let $\mathtt{P}_\mathrm{I},\mathtt{P}_\mathrm{II'},\mathtt{P}_\mathrm{III}$ be as above. Then 
	\begin{equation}
	\mathtt{P}_\mathrm{I},\mathtt{P}_\mathrm{II'},\mathtt{P}_\mathrm{III} \in \calT, 
	\end{equation}
	where $\calT \subseteq\operatorname{Borel}(\scrX^\bbN)$ is the ``tail $\sigma$-algebra'' $\calT = \cap_{N\in \bbN} \pi_N^* \calP$. Consequently, given any $\bbK$-valued random variables $\lambda_0,\lambda_1,\lambda_2,\cdots : \Omega\to \bbK$, the random formal series $\Sigma:\Omega\to \scrX^\bbN$ given by $\Sigma(\omega) = \sum_{n=0}^\infty \lambda_n(\omega) x_n$  is such that 
	\begin{equation}
	\Sigma^{-1}(\mathtt{P}_i) \in \cap_{N\in \bbN} \sigma(\{\lambda_n\}_{n=N}^\infty)
	\label{eq:SF}
	\end{equation}
	for each $i\in \{\mathrm{I},\mathrm{II'},\mathrm{III}\}$.
\end{lemma}
\begin{proof}
	Clearly, $\pi_N^{-1}(\mathtt{P}_{i}) = \mathtt{P}_i$ for each $i\in \{\mathrm{I},\mathrm{II'},\mathrm{III}\}$. By \Cref{lem:meas}, we can therefore conclude that $\mathtt{P}_i \in \calT$. If $\Sigma$ is as above, then $\Sigma^*\circ \pi_N^* \calP \subseteq \sigma(\{\lambda_n\}_{n=N}^\infty) $. Since $\Sigma^{-1}(\mathtt{P}_i)$ is in the left-hand side for each $N\in \bbN$, \cref{eq:SF} follows. 
\end{proof}

\begin{proposition}\label{prop:kol}
	Let $f:\bbN\to \bbN$ satisfy $|f^{-1}(\{n\})|<\infty$ for all $n\in \bbN$. 
	Suppose that $\lambda_0,\lambda_1,\lambda_2,\cdots:\Omega \to \bbK$ are independent random variables on the probability space $(\Omega,\calF,\bbP)$, and consider the random formal series $\Sigma:\Omega\to \scrX^\bbN$ given by 
	\begin{equation}
	\Sigma(\omega) = \sum_{n=0}^\infty \lambda_{f(n)}(\omega) x_n.
	\end{equation}
	Then $\bbP (\Sigma^{-1}(\mathtt{P})) = \bbP[\Sigma \in \mathtt{P}] \in \{0,1\}$ for any element  $\mathtt{P} \in \calT$, and in particular for the sets $\mathtt{P}_i$ for each $i\in \{\mathrm{I},\mathrm{II'},\mathrm{III}\}$. 
\end{proposition}
\begin{proof}
	Since $\lambda_0,\lambda_1,\lambda_2,\cdots$ are now assumed to be independent, that $\bbP[\Sigma \in \mathtt{P}] \in \{0,1\}$ follows immediately from the Kolmogorov zero-one law \cite[Theorem 2.5.3]{Durrett}. 
	By \Cref{lem:01}, this applies to $\mathtt{P}_{\mathrm{I}},\mathtt{P}_{\mathrm{II}'},\mathtt{P}_{\mathrm{III}}$.  
\end{proof}

\begin{proposition}\label{prop:eq}
	Let $f:\bbN\to \bbN$ satisfy $|f^{-1}(\{n\})|<\infty$ for all $n\in \bbN$. 
	Suppose that $\mathtt{P}\subseteq \scrX^\bbN$ is a $\bbK$-subspace and that $\zeta_0,\zeta_1,\zeta_2,\cdots:\Omega\to \bbK$ are a collection of symmetric, independent $\bbK$-valued random variables. 
	
	Then, letting $\Sigma,S:\Omega\to \scrX^\bbN$ denote the random formal series 
	\begin{equation}
	\Sigma(\omega)  = \sum_{n=0}^\infty \zeta_{f(n)}(\omega) x_n \quad \text{ and }\quad S(\omega) = \sum_{n=0}^\infty \chi_{f(n)}(\omega) x_n,
	\end{equation}
	where $\chi_n = 2^{-1}(1-\zeta_n)$, 
	the following are  equivalent: ($*$) $\Sigma \in \mathtt{P}$ for $\bbP$-almost all $\omega\in \Omega$ \emph{and} $\sum_{n=0}^\infty x_n \in \mathtt{P}$, ($**$) $S \in \mathtt{P}$ for $\bbP$-almost all $\omega\in \Omega$. 
	Consequently, if $\mathtt{P}\in \calT$, by \Cref{prop:kol} the following are equivalent: ($*'$) $\Sigma \not\in \mathtt{P}$ for $\bbP$-almost all $\omega\in \Omega$ \emph{or} $\sum_{n=0}^\infty x_n \not\in \mathtt{P}$ and ($**'$) $S \not\in \mathtt{P}$ for $\bbP$-almost all $\omega\in \Omega$. 
\end{proposition}
This is essentially an immediate consequence of \cref{eq:misc_j21}, \cref{eq:misc_j22}, \emph{mutatis mutandis}.
\begin{proof}
	First suppose that ($*$) holds. In particular, $\sum_{n=0}^\infty x_n \in \mathtt{P}$. Then, since $\mathtt{P}$ is a subspace of $\scrX^\bbN$, 
	\begin{equation}
	\sum_{n=0}^\infty \chi_{f(n)}(\omega) x_n = -\frac{1}{2}\sum_{n=0}^\infty \zeta_{f(n)}(\omega) x_n + \frac{1}{2} \sum_{n=0}^\infty x_n 
	\end{equation}
	is in $\mathtt{P}$ if $\sum_{n=0}^\infty \zeta_n(\omega) x_n$ is. By assumption, this holds for $\bbP$-almost all $\omega\in\Omega$, and so we conclude that ($**$) holds.
	
	Conversely, suppose that ($**$) holds, so that $S(\omega)\in \mathtt{P}$ for all $\omega$ in some some subset $F\in \calF$ with $\bbP(F)=1$. Clearly, the two formal series $S,S':\Omega\to \scrX^\bbN$, 
	\begin{equation}
	S(\omega) = \sum_{n=0}^\infty \chi_{f(n)}(\omega)x_n \quad \text{ and } \quad S'(\omega) = \sum_{n=0}^\infty (1-\chi_{f(n)}(\omega)) x_n 
	\end{equation}
	are equidistributed. We deduce that $S'(\omega)\in \mathtt{P}$ for almost all $\omega\in \Omega$, i.e.\ that there exists some $F'\in \calF$ with $\bbP(F')= 1$ such that $S'(\omega)\in \mathtt{P}$ whenever $\omega\in F'$. This implies, since $\mathtt{P}$ is a subspace of $\scrX^\bbN$, that the random formal series 
	\begin{align}
	S(\omega)+S'(\omega)  &= \sum_{n=0}^\infty x_n  \\ S(\omega) - S'(\omega) &= -\sum_{n=0}^\infty \zeta_{f(n)}(\omega) x_n 
	\end{align}
	are both in $\mathtt{P}$ for all $\omega\in F\cap F'$. Since $\bbP(F\cap F') = 1$, it is the case that $F\cap F' \neq \varnothing$, and so we conclude that $\sum_{n=0}^\infty x_n \in \mathtt{P}$. Likewise, $\sum_{n=0}^\infty \zeta_{f(n)}(\omega) x_n \in \mathtt{P}$ for almost all $\omega\in \Omega$. 
\end{proof}
\Cref{prop:eq} applies in particular to the sets $\mathtt{P}_{\mathrm{I}},\mathtt{P}_{\mathrm{II'}},\mathtt{P}_{\mathrm{III}}$. We will not discuss $\mathtt{P}_{\mathrm{III}}$ further, but the preceding results are useful for the treatment of the J{\o}rgensen--Kwapie\'n and Bessaga--Pe{\l}czy\'nski theorems along the lines of \S\ref{sec:main}.

\section{Proof of It\^o--Nisio}
\label{sec:IN}

Let $\scrX$ be a separable Banach space over $\bbK \in \{\bbR,\bbC\}$. We now give a treatment, 
via the method in \cite{ABS}, of the particular variant of the It\^o--Nisio theorem stated in \Cref{thm:IN}.

The key result allowing the generalization from the weak topology to all admissible topologies is:
\begin{proposition}	\label{prop:con}
	If $\tau$ is an admissible topology on $\scrX$, then $\operatorname{Borel}(\scrX) = \operatorname{Borel}(\scrX_\tau)$. 
\end{proposition}
\begin{proof}
	The inclusion $\operatorname{Borel}(\scrX) \supseteq \operatorname{Borel}(\scrX_\tau)$ is an immediate consequence of the assumption that $\tau$ is weaker than or identical to the norm topology, so it suffices to prove that $\operatorname{Borel}(\scrX_\tau)$ contains a collection of sets that generate $\operatorname{Borel}(\scrX)$ as a $\sigma$-algebra. Consider the collection 
	\begin{equation}
		\calB=\{x+\lambda \bbB : x\in \scrX, \lambda \in \bbR^{\geq 0}\} \subseteq\operatorname{Borel}(\scrX)
	\end{equation}
	of all norm-closed balls in $\scrX$. 
	Since $\scrX$ is separable, the collection of all open balls generates $\operatorname{Borel}(\scrX)$, and each open ball $x+\lambda \bbB^\circ$, $x\in \scrX,\lambda >0$, is a countable union 
	\begin{equation} 
		x+\lambda \bbB^\circ = \bigcup_{N\in \bbN, 1/N<\lambda} (x+(\lambda-1/N) \bbB)
	\end{equation} 
	of closed balls, so the closed balls generate $\operatorname{Borel}(\scrX)$. 
	Since $\tau$ is an LCTVS topology, once we know that $\bbB$ is $\tau$-closed, the same holds for all other norm-closed balls.  
	Because $\tau$ is admissible, the elements of $\calB$ are $\tau$-closed, so $\calB\subseteq\operatorname{Borel}(\scrX_\tau)$. 
\end{proof}

Suppose now that $\tau$ is admissible, and suppose that $(\Omega,\calF,\bbP)$ is a probability space on which symmetric, independent random variables $\gamma_0,\gamma_1,\gamma_2,\cdots:\Omega\to \bbK$ are defined. 
\begin{proposition} \label{prop:rv}
	 Suppose that $\sum_{n=0}^\infty \gamma_n(\omega) x_n$ converges in $\scrX_\tau$ for $\bbP$-almost all $\omega\in \Omega$, so that we may find some $F\in \calF$ with $\bbP(F)=1$ such that 
	\begin{equation}
		\Sigma_\infty(\omega) = \tau\!-\!\!\!\lim_{N\to\infty} \sum_{n=0}^N \gamma_n(\omega) x_n
		\label{eq:cnv}
	\end{equation}
	exists
	for all $\omega\in F$. Set $\Sigma_\infty(\omega)=0$ for all $\omega \in \Omega\backslash F$. Then, $\Sigma_\infty$ is a well-defined $\scrX$-valued random variable.
\end{proposition}
\begin{proof}
	We want to prove that $\Sigma_\infty$ is measurable with respect to $\calF$ and $\operatorname{Borel}(\scrX)$. By \Cref{prop:con} and \Cref{lem:weak}, $\operatorname{Borel}(\scrX) = \operatorname{Borel}(\scrX_\tau)= \operatorname{Borel}(\sigma(\scrX,\scrX_\tau^*)) = \sigma(\scrX_\tau^*)$, so it suffices to check that $\Lambda\circ \Sigma_\infty$ is a measurable $\bbK$-valued function for each $\Lambda \in \scrX_\tau^*$. 
	Certainly, 
	\begin{equation}
		\Lambda\circ \tilde{\Sigma}_N(\omega) = 1_{\omega\in F}\Lambda\circ  \Sigma_N(\omega) =
		\begin{cases}
			\Sigma_N(\omega) & (\omega\in F) \\ 
			0 & (\omega \in \Omega\backslash F)
		\end{cases}
	\end{equation}
	is measurable. Consequently, $\Lambda\circ \Sigma_\infty = \lim_{N\to\infty} \Lambda\circ \tilde{\Sigma}_N$ is the limit of measurable $\bbK$-valued random variables and, therefore, measurable. 
\end{proof}

\begin{proposition}	\label{prop:equid}
	Consider the setup of \Cref{prop:rv}. For each $N\in \bbN$, the $\scrX$-valued random variables $\Sigma_\infty$ and $\Sigma_\infty - 2 \Sigma_N$ are equidistributed. 
\end{proposition}
\begin{proof}
	Denote the laws $\Sigma_\infty,\Sigma_\infty-2\Sigma_N$ by $\mu, \lambda_N:\operatorname{Borel}(\scrX)\to [0,1]$, respectively. The measures $\mu,\lambda_N$ are uniquely determined by their Fourier transforms $\calF \mu,\calF \lambda_N : \scrX^*_\tau \to \bbC$, 
	\begin{equation}
		\calF\mu(\Lambda) = \int_{\Omega} e^{-i \Lambda \Sigma_\infty(\omega)} \dd \bbP(\omega) = \int_{\scrX} e^{-i\Lambda x} \dd \mu(x) ,
	\end{equation}
	where $\calF\lambda_N$ is defined analogously. 
	For each $\Lambda \in \scrX^*_\tau$, $\Lambda (\Sigma_\infty - \Sigma_N)$ and $\Lambda(\Sigma_N)$ are clearly independent, and $\Lambda(\Sigma_N)$ is equidistributed with $-\Lambda(\Sigma_N)$, so 
	\begin{align}
		\begin{split} 
			\calF\mu(\Lambda) = \int_{\Omega} e^{-i \Lambda \Sigma_\infty(\omega)} \dd \bbP(\omega) &= \int_{\Omega} e^{-i \Lambda (\Sigma_\infty(\omega) - \Sigma_N(\omega))} e^{-i \Lambda  \Sigma_N(\omega)} \dd \bbP(\omega) \\
			&=\Big(\int_{\Omega} e^{-i \Lambda (\Sigma_\infty(\omega) - \Sigma_N(\omega))} \dd \bbP(\omega) \Big) \Big(\int_\Omega e^{-i \Lambda  \Sigma_N(\omega)} \dd \bbP(\omega)\Big) \\
			&=\Big(\int_{\Omega} e^{-i \Lambda (\Sigma_\infty(\omega) - \Sigma_N(\omega))} \dd \bbP(\omega) \Big) \Big(\int_\Omega e^{+i \Lambda  \Sigma_N(\omega)} \dd \bbP(\omega)\Big) \\
			&= \int_{\Omega} e^{-i \Lambda (\Sigma_\infty(\omega) - \Sigma_N(\omega))} e^{+i \Lambda  \Sigma_N(\omega)} \dd \bbP(\omega) \\
			&= \int_{\Omega} e^{-i \Lambda (\Sigma_\infty(\omega) - 2\Sigma_N(\omega))}  \dd \bbP(\omega) = \calF \lambda_N(\Lambda).
		\end{split} 
	\end{align}
	Hence the Fourier transforms of $\mu,\lambda_N$ agree, and we conclude that $\Sigma_\infty$ and $\Sigma_\infty-2\Sigma_N$ are equidistributed. 
\end{proof}
The proof is identical to the standard one, except we need to know that the law of an $\scrX$-valued random variable is uniquely determined by the restriction of its Fourier transform (a.k.a.\ ``characteristic functional'') from $\scrX^*$ to $\scrX_\tau^*$, for any admissible $\tau$. 
The proof of this fact for $\tau$ the strong or weak topologies, which is just the proof that a finite Borel measure on $\scrX$ is uniquely determined by the Fourier transform of its law, is given in \cite[E.1.16, E.1.17]{ABS}.
The general statement follows from analogous reasoning: the finite-dimensional version (i.e.\ finite Borel measures on $\bbR^d$ are identifiable with particular tempered distributions, and are, therefore, uniquely determined by their Fourier transforms), the Dynkin $\pi$-$\lambda$ theorem (which implies that a finite measure is uniquely determined by its restriction to any $\pi$-system which generates the $\sigma$-algebra on which the measure is defined \cite[Theorem A.1.5]{Durrett}), and \Cref{prop:con}. 

Another way to prove the proposition is to show that $\Sigma_\infty$ agrees, almost everywhere, with the composition of the random formal series $\sum_{n=0}^\infty \gamma_n(-)x_n:\Omega\to \scrX^\bbN$ and $\Sigma_{\infty,\mathrm{Uni}}:\scrX^\bbN\to \scrX$,
\begin{equation}
	\Sigma_{\infty,\mathrm{Uni}}\Big( \sum_{n=0}^\infty x_n \Big) = 
	\begin{cases}
		\calS\!-\!\lim_{N\to\infty} \sum_{n=0}^N x_n & (\sum_{n=0}^\infty x_n \in \mathtt{P}_{\mathrm{II}'}),  \\
		0 & (\text{otherwise}),
	\end{cases}
\end{equation}
where $\calS\subseteq\scrX_\tau^*$ is a countable norming collection of functionals and $\mathtt{P}_{\mathrm{II}'}$ is as in \S\ref{sec:preliminaries}. By the results in \S\ref{sec:preliminaries}, $\Sigma_{\infty,\mathrm{Uni}}:\scrX^\bbN\to \scrX$ is Borel measurable. Thus, we can form the pushforward under it of the law of the formal series $\sum_{n=0}^\infty \gamma_n(-)x_n$. The initial claim, then, is that the law of $\Sigma_\infty$ is this pushforwards. 
Likewise, the pushforwards of the law of the random formal series
\begin{equation}
	\omega\mapsto -\sum_{n=0}^N \gamma_n(\omega)x_n + \sum_{n=N+1}^\infty \gamma_n(\omega)x_n \in \scrX^\bbN \label{eq:misc_g31}
\end{equation}
is the law of $\Sigma_\infty - 2 \Sigma_N$. 
Since the random formal series \cref{eq:misc_g31} is equidistributed with the original, we deduce that $\Sigma_\infty$ and $\Sigma_{\infty}-2\Sigma_N$ are equidistributed as well.

Recall that an $\scrX$-valued random variable $X:\Omega\to \scrX$ is called \emph{tight} if for every $\varepsilon>0$ there exists a norm-compact set $K\subseteq\scrX$ such that $\bbP[X\notin K] \leq \varepsilon$. By an elementary argument, every $\scrX$-valued random variable is tight \cite[Proposition 6.4.5]{ABS}. 
A family $\calX$ of $\scrX$-valued random variables is called \emph{uniformly tight} if we can choose the same $K=K(\varepsilon)$ for every $X\in \calX$, i.e.\ if for each $\varepsilon>0$ there exists some norm-compact $K\subseteq\scrX$  such that $\bbP[X\notin K] \leq \varepsilon$ holds for all $X\in \calX$. 
If $\calX$ is uniformly tight, then 
\begin{equation} 
\calX-\calX = \{X_1-X_2 : X_1,X_2 \in \calX\}
\label{eq:XdX}
\end{equation} 
is uniformly tight as well, a fact which is used below.  (The map $\Delta:\scrX\times \scrX\to \scrX$ given by $(x,y)\mapsto x-y$ is continuous. If $K\subseteq\scrX$ is compact, then $K\times K$ is a compact subset of $\scrX\times \scrX$. Its image $\Delta(K\times K) = K-K$ under $\Delta$ is, therefore, also compact. By a union bound, 
\begin{equation} 
	\bbP[X_1-X_2 \notin \Delta(K\times K)] \leq \bbP[X_1 \notin K] + \bbP[X_2\notin K].
\end{equation} 
See \cite[Lemma 6.4.6]{ABS}.)

To complete the proof of the It\^o--Nisio theorem, we use L\'evy's maximal inequality \cite[Proposition 6.1.12]{ABS}\footnote{The statement there uses strict inequalities for the events, but the version for nonstrict inequalities follows by the countable additivity of $\bbP$.}:
\begin{propositionp}[L\'evy's maximal inequality]
	Let $\scrX$ be a separable Banach space over $\bbK$. 
	Let $x_0,x_1,x_2,\cdots$ be independent symmetric $\scrX$-valued random variables. Then, setting $\smash{\Sigma_N  = \sum_{n=0}^N x_n}$, 
	\begin{equation}
		\bbP[(\exists N_0 \in \{0,\cdots,N\})\lVert \Sigma_{N_0} \rVert \geq  R] \leq 2 \bbP[\lVert \Sigma_N \rVert \geq R]
	\end{equation}
	for all $N\in \bbN$ and real $R>0$.
\end{propositionp}

\begin{proposition}
		Suppose that $\sum_{n=0}^\infty \gamma_n(\omega) x_n$ converges in $\scrX_\tau$ for $\bbP$-almost all $\omega\in \Omega$, and let $\Sigma_\infty$ denote the $\scrX$-valued random variable constructed in the statement of \Cref{prop:rv}. Then 
		\begin{equation}
			\Sigma_\infty(\omega) = \lim_{N\to\infty} \sum_{n=0}^N \gamma_n(\omega) x_n
			\label{eq:cvs}
		\end{equation}
		for $\bbP$-almost all $\omega\in \Omega$. 
		\label{prop:IN}
\end{proposition}
The limit here is taken in the strong topology.
\begin{proof}
	The proof is split into three parts. 
	We first show that it suffices to show that $\Sigma_N\to \Sigma_\infty$ in probability,  where $\smash{\Sigma_N=\sum_{n=0}^N \gamma_n(\omega)x_n}$, i.e.\ that 
	\begin{equation}
		\lim_{N\to\infty} \bbP[\lVert \Sigma_\infty - \Sigma_N \rVert> \varepsilon] = 0
	\end{equation}
	for all $\varepsilon>0$. This part of the argument uses L\'evy's inequality. 
	We then establish (via a standard trick) the uniform tightness of $\{\Sigma_N\}_{N=0}^\infty$.  The third step involves showing that, if $\Sigma_N$ fails to converge to $\Sigma_\infty$ in probability, then, with positive probability, $\Sigma_N$ fails to converge to $\Sigma_\infty$ in $\scrX_\tau$. Under our assumption to the contrary, we can then conclude that $\Sigma_N\to \Sigma_\infty$ in probability, which by the first part of the argument completes the proof of the proposition. 
	\begin{enumerate}
		\item Suppose that $\lim_{N\to\infty} \bbP[\lVert \Sigma_\infty - \Sigma_N \rVert> \varepsilon] = 0$ for all $\varepsilon>0$. We want to prove that $\Sigma_N\to \Sigma_\infty$ $\bbP$-almost surely. It suffices to prove that $\{\Sigma_N\}_{N=0}^\infty$ is $\bbP$-almost surely Cauchy, since then by the completeness of $\scrX$ it converges strongly $\bbP$-almost surely  to some random limit $\Sigma_\infty':\Omega\to \scrX$. Since the $\tau$ topology is weaker than (or identical to) the strong topology and Hausdorff, $\Sigma_\infty'=\Sigma_\infty$ $\bbP$-almost surely. 
		
		By the triangle inequality, for any $M,M',N \in \bbN$, $\lVert \Sigma_{M}- \Sigma_{M'} \rVert \leq \lVert \Sigma_{M}- \Sigma_{N} \rVert+\lVert \Sigma_{M'}- \Sigma_{N} \rVert$. Therefore, by a union bound, 
		\begin{equation}
			\bbP\Big[\bigcup_{ M,M' \geq N} \lVert \Sigma_{M}- \Sigma_{M'} \rVert  \geq \varepsilon \Big] \leq 2\bbP\Big[\bigcup_{ M \geq N} \lVert \Sigma_{M}- \Sigma_{N} \rVert  \geq \varepsilon/2\Big].
		\end{equation}
		By the countable additivity of $\bbP$ and by L\'evy's maximal inequality, 
		\begin{align}
			2\bbP\Big[\bigcup_{ M \geq N} \lVert \Sigma_{M}- \Sigma_{N} \rVert  \geq \varepsilon/2\Big] &= \lim_{N'\to \infty} 2\bbP\Big[\bigcup_{N'\geq M \geq N} \lVert \Sigma_{M}- \Sigma_{N} \rVert  \geq \varepsilon/2\Big] \\
			&\leq  \lim_{N'\to \infty} 4\bbP\Big[\lVert \Sigma_{N'}- \Sigma_{N} \rVert  \geq \varepsilon/2\Big].
		\end{align}
		Consequently, 
		\begin{align}
				\label{eq:nca}
				\begin{split} 
			 	\bbP\Big[\bigcup_{\varepsilon>0} \bigcap_{N=0}^\infty \bigcup_{ M,M' \geq N} \lVert \Sigma_{M}- \Sigma_{M'} \rVert  \geq \varepsilon \Big]  &= \lim_{\varepsilon\to 0^+} \lim_{N\to\infty} 	\bbP\Big[\bigcup_{ M,M' \geq N} \lVert \Sigma_{M}- \Sigma_{M'} \rVert  \geq \varepsilon \Big] \\
			 	&\leq 4  \lim_{\varepsilon\to 0^+} \lim_{N\to\infty} \lim_{N'\to\infty} \bbP[\lVert \Sigma_{N'}- \Sigma_{N} \rVert  \geq \varepsilon/2].
			 	\end{split}
		\end{align}
		By the triangle inequality and a union bound, 
		\begin{equation} 
		\bbP[\lVert \Sigma_{N'}- \Sigma_{N} \rVert  \geq \varepsilon/2]\leq \bbP[\lVert \Sigma_{\infty}- \Sigma_{N} \rVert  \geq \varepsilon/4] + \bbP[\lVert \Sigma_{N'}- \Sigma_{\infty} \rVert  \geq \varepsilon/4]. 
		\end{equation}
		It follows from the assumption that $\Sigma_N\to \Sigma_\infty$ in probability that
		\begin{equation} 
		\lim_{N\to\infty} \lim_{N'\to\infty} \bbP[\lVert \Sigma_{N'}- \Sigma_{N} \rVert  \geq \varepsilon/2] = 0.
		\end{equation} 
		Consequently, the right-hand side and thus left-hand side of \cref{eq:nca} are zero. 
		The event on the left-hand side of \cref{eq:nca} is the event that the sequence $\{\Sigma_N\}_{N=0}^\infty$ fails to be Cauchy, so the preceding argument shows that $\{\Sigma_N(\omega)\}_{N=0}^\infty$ is Cauchy for $\bbP$-almost all $\omega\in \Omega$. 
		
		\item By \Cref{prop:equid}, $\Sigma_\infty$ and $\Sigma_\infty - 2 \Sigma_N$ are equidistributed, for each $N\in \bbN$. 
		For any $\varepsilon>0$, by the (automatic) tightness of $\Sigma_\infty$ there is a norm-compact subset $K\subseteq\scrX$ such that $\bbP[\Sigma_\infty \notin K] < \varepsilon$. Let $L=(1/2)(K-K)$, which is also compact. Then, by a union bound,
		\begin{equation}
			\bbP[\Sigma_N \notin L] \leq \bbP[\Sigma_\infty \notin K] + \bbP[\Sigma_\infty - 2\Sigma_N \notin K] = 2 \bbP[\Sigma_\infty \notin K] < 2 \varepsilon. 
		\end{equation}
		We conclude that $\{\Sigma_N\}_{N=0}^\infty$ is uniformly tight.
		
		Also, since $\Sigma_\infty$ is tight, the family $\calX=\{\Sigma_N\}_{N=0}^\infty \cup \{\Sigma_\infty\}$ is uniformly tight, which implies that the family $\{\Sigma_\infty - \Sigma_N\}_{N=0}^\infty \subseteq\calX-\calX $ is uniformly tight. Consequently, there exists for each $\varepsilon>0$ a norm-compact subset $K_0=K_0(\varepsilon)\subseteq\scrX$ such that 
		\begin{equation} 
			\bbP[(\Sigma_\infty - \Sigma_N) \notin K_0(\varepsilon)] \leq \varepsilon 
			\label{eq:k0d}
		\end{equation} 
		for all $N\in \bbN$. 
		\item Suppose that $\Sigma_N$ does not converge to $\Sigma_\infty$ in probability, so that there exist some $\varepsilon,\delta>0$ and some subsequence $\{\Sigma_{N_k}\}_{k=0}^\infty \subseteq\{\Sigma_N\}_{N=0}^\infty$ such that 
		\begin{equation}
			\bbP[\lVert \Sigma_\infty - \Sigma_{N_k} \rVert > \varepsilon] \geq \delta
			\label{eq:k39} 
		\end{equation}
	for all $k\in \bbN$. Consider the set $K_0=K_0(\delta/2)$ defined in \cref{eq:k0d}, so that $\bbP[(\Sigma_\infty - \Sigma_N) \notin K_0] \leq \delta/2$ for all $N\in \bbN$. Then, combining this inequality with the inequality \cref{eq:k39}, $\bbP[ (\Sigma_\infty - \Sigma_{N_k}) \in K_0\backslash \varepsilon \bbB] \geq \delta/2$ for all $k\in \bbN$. It follows that the quantity 
	\begin{align}
		\bbP[ (\Sigma_\infty - \Sigma_{N_k}) \in K_0 \backslash \varepsilon \bbB\text{ i.o.}] &= \bbP[ \cap_{K\in \bbN} \cup_{k\geq K}(\Sigma_\infty - \Sigma_{N_k}) \in K_0\backslash \varepsilon \bbB] \\
		&= \lim_{K\to\infty} \bbP[  \cup_{k\geq K}(\Sigma_\infty - \Sigma_{N_k}) \in K_0\backslash \varepsilon \bbB] 
	\end{align}
	(where ``i.o.''\ means for infinitely many $k$)
	is bounded below by $\delta/2$ and is in particular positive. So, for $\omega$ in some set of positive probability, there exists an $\omega$-dependent subsequence $\{N'_\kappa(\omega)\}_{\kappa=0}^\infty = \{N_{k_\kappa}(\omega)\}_{\kappa=0}^\infty$ such that $\Sigma_\infty(\omega) - \Sigma_{N'_\kappa}(\omega) \in K_0\backslash \varepsilon \bbB$ for all $\kappa \in \bbN$. 
	
	Since $K_0$ is a compact subset of a metric space, it is sequentially compact, so by passing to a further subsequence we can assume without loss of generality that $\Sigma_\infty(\omega) - \Sigma_{N'_\kappa}(\omega)$ converges strongly to some $\omega$-dependent $\Delta(\omega) \in \scrX$, for $\omega$ in some subset of positive probability. 
	But, for such $\omega$, $\lVert \Delta(\omega)\rVert\geq \varepsilon$ necessarily, so  $\Delta(\omega) \neq 0$. 
	Since $\tau$ is weaker than or identical to the strong topology, 
	\begin{equation} 
		(\Sigma_\infty(\omega) - \Sigma_{N_\kappa'}(\omega)) \to \Delta(\omega) \neq 0
	\end{equation} 
	in $\scrX_\tau$ for such $\omega$. Since $\tau$ is Hausdorff, $\Sigma_N(\omega)$ does not $\tau$-converge to $\Sigma_\infty(\omega)$ as $N\to\infty$. 
	\end{enumerate}
We conclude that (\ref{eq:cvs}) holds for $\bbP$-almost all $\omega\in\Omega$ under the hypotheses of the proposition. 
\end{proof}

It is clear that which of the cases in \Cref{thm:IN} hold depends only on $\{x_n\}_{n=0}^\infty$ and the laws of the random variables $\gamma_0,\gamma_1,\gamma_2,\cdots$.

\section{Proof of Orlicz--Pettis}
\label{sec:main}

Let $\scrX$ be a separable Banach space over $\bbK\in \{\bbR,\bbC\}$, and let $\tau$ be an admissible topology on it.

\begin{proposition} \label{prop:POP}
	Suppose that $\zeta_0,\zeta_1,\zeta_2,\cdots:\Omega\to \bbK$ are a collection of symmetric, independent $\bbK$-valued random variables such that, for some infinite $\calT\subseteq\bbN$,
	\begin{equation}
	\bbP[\exists \varepsilon>0 \text{ s.t. }|\zeta_n| >\varepsilon  \text{ for infinitely many }n\in \calT] = 1.
	\label{eq:misc_gg3}
	\end{equation}
	Suppose further that $\{X_n\}_{n=0}^\infty \in \scrX^\bbN$ is some sequence satisfying
	\begin{equation} 
	\inf_{n\in \calT} \lVert X_n \rVert>0.
	\label{eq:misc_h41}
	\end{equation}
	Then, for any $\calT_0\subseteq\bbN$ such that $\calT_0\supseteq \calT$, it is the case that, for $\bbP$-almost all $\omega \in \Omega$, the sequence $\{\Sigma_N(\omega)\}_{N=0}^\infty$ given by 
	\begin{equation}
	\Sigma_N(\omega) = \sum_{n=0,n\in \calT_0}^N \zeta_n(\omega) X_n
	\label{eq:Sn3}
	\end{equation}
	fails to $\tau$-converge as $N\to\infty$. Therefore, the random formal series $\Sigma : \Omega\to \scrX^\bbN$ defined by $\Sigma(\omega)=\sum_{n=0}^\infty 1_{n\in \calT_0} \zeta_n(\omega)X_n$ satisfies $\Sigma(\omega) \not\in \mathtt{P}_{\mathrm{II}}$ for $\bbP$-almost all $\omega\in \Omega$.
\end{proposition}
\begin{proof}
	By \Cref{prop:kol} and the inclusion $\mathtt{P}_{\mathrm{II'}} \supset \mathtt{P}_{\mathrm{II}}$ (where $\mathtt{P}_{\mathrm{II'}}$ is as in \S\ref{sec:preliminaries}), it suffices to prove that it is not the case that $\Sigma(\omega) = \sum_{n=0}^\infty 1_{n\in \calT_0} \zeta_n(\omega)X_n$ is $\bbP$-almost surely $\calS$-weakly summable, where $\calS\subseteq\scrX_\tau^*$ is a countable collection of norming functionals. 
	Suppose, to the contrary, that $\Sigma$ were almost surely $\calS$-weakly summable. By the It\^o--Nisio theorem, this would imply that $\{\Sigma_N(\omega)\}_{N=0}^\infty$ converges strongly for $\bbP$-almost all $\omega\in \Omega$. But, the conjunction of \cref{eq:misc_gg3} and $\inf_{n\in \calT} \lVert X_n \rVert>0$ implies instead that $\{\Sigma_N(\omega)\}_{N=0}^\infty$ almost surely \emph{fails} to converge strongly. 
\end{proof}

\begin{proposition}
	Let $f:\bbN\to \bbN$. If it is the case that 
	\begin{equation}
	\tau\!-\!\!\lim_{N\to\infty}\sum_{n=0}^N \epsilon_{f(n)}(\omega)x_n 
	\label{eq:misc_h44}
	\end{equation}
	exists for $\bbP$-almost all $\omega\in \Omega$, then, for any subset $\calT\subseteq\bbN$,  
	\begin{equation} 
	\tau\!-\!\!\lim_{N\to\infty}\sum_{n=0, f(n)\in \calT}^N \epsilon_{f(n)}(\omega)x_n 
	\label{eq:misc_tt1}
	\end{equation}
	exists for $\bbP$-almost all $\omega\in \Omega$.
\end{proposition}
\begin{proof}
	Let 
	\begin{equation}
	\epsilon_n' = 
	\begin{cases}
	\epsilon_n & (n\notin \calT) \\ 
	-\epsilon_n & (n\in \calT). 
	\end{cases}
	\end{equation}
	We can now consider the random formal series 
	\begin{align}
	\sum_{n=0}^\infty (\epsilon_{f(n)}' - \epsilon_{f(n)}) x_n  &= \sum_{n=0}^\infty \epsilon_{f(n)}' x_n -\sum_{n=0}^\infty \epsilon_{f(n)} x_n  \label{eq:misc_h11}
	\\ &= 2\sum_{n=0,f(n)\in \calT}^\infty \epsilon_{f(n)}x_n. \label{eq:misc_h22}
	\end{align}
	The two random formal series on the right-hand side of \cref{eq:misc_h11} are equidistributed, so, under the hypothesis of the proposition, both are $\tau$-summable for $\bbP$-almost all $\omega\in \Omega$. Thus, the formal series on the right-hand side of \cref{eq:misc_h22} is $\bbP$-almost surely $\tau$-summable. 
\end{proof}

We deduce \Cref{thm:PO} (and thus \Cref{thm:OP}) as a corollary of the previous two propositions. We prove the slightly strengthened claim that, for $\bbP_{\mathrm{Haar}}$-almost all $\{\epsilon_n\}_{n=0}^\infty \in \{-1,+1\}^\bbN$, the formal series in \cref{eq:misc_h31}
both fail to even be $\calS$-weakly summable. By \Cref{prop:eq}, we just need to show that it is \emph{not} the case that, for $\bbP_{\mathrm{Haar}}$-almost all $\{\epsilon_n\}_{n=0}^\infty\in \{-1,+1\}^\bbN$, the formal series 
\begin{equation}
	\sum_{n=0,f(n)\in \calT}^\infty \epsilon_{f(n)} x_n \in \scrX^\bbN 
	\label{eq:misc_g11}
\end{equation}
is $\calS$-weakly summable. Suppose, to the contrary, that it is $\calS$-weakly summable for $\bbP_{\mathrm{Haar}}$-almost all $\{\epsilon_n\}_{n=0}^\infty$. Owing in part to the assumption that $|f^{-1}(\{n\})|<\infty$ for all $n\in \bbN$ (along with \cref{eq:misc_g41}), there exists a $\calT_0\subseteq\calT$ such that 
\begin{itemize}
	\item $f:f^{-1}(\calT_0)\to \bbN$ is monotone and
	\item $	\inf_{n\in \calT_0} \rVert \sum_{ n_0 \in f^{-1}(\{n\}) } x_{n_0} \rVert >0$.
\end{itemize}
By the previous proposition, $\sum_{n=0,f(n)\in \calT_0}^\infty \epsilon_{f(n)} x_n \in \scrX^\bbN$
is $\calS$-weakly summable 
$\bbP$-almost surely. Since $f|_{f^{-1}(\calT_0)}$ is monotone, we deduce that 
\begin{equation}
	\sum_{n=0,n\in \calT_0}^\infty \epsilon_{n} \Big[ \sum_{n_0\in f^{-1}(\{n\})} x_{n_0} \Big] \in \scrX^\bbN 
	\label{eq:misc_g13}
\end{equation}
is $\calS$-weakly summable $\bbP$-almost surely. However, this contradicts \Cref{prop:POP}.

\section*{Acknowledgements}

This work was partially supported by a Hertz fellowship. The author would like to thank the reviewer for their comments.

\appendix

\section{Admissible topologies}
Let $\scrX$ denote a Banach space over $\bbK\in \{\bbR,\bbC\}$, and let $\tau$ be an admissible topology on it. 
\begin{lemma}
	The $\tau$-weak topology, a.k.a.\ the $\sigma(\scrX,\scrX_\tau^*)$-topology, is admissible.
	\label{lem:weak}
\end{lemma}
\begin{proof} \hfill 
	\begin{enumerate}
		\item The $\tau$-weak topology is an LCTVS-topology on $\scrX$ \cite[\S3.10, \S3.11]{Rudin} identical to or weaker than the norm topology. 
		
		For each $\Lambda \in \scrX_\tau^*$ and closed interval $I\subseteq[-\infty,+\infty]$, let $C_{\Lambda,I}$ denote the $\tau$-weakly closed subset (I) $C_{\Lambda,I} = \Lambda^{-1}(I)$ if $\bbK=\bbR$ or (II) $C_{\Lambda,I} = \Lambda^{-1}(\{z\in \bbC: \Re z \in I\})$ otherwise. 
		By the Hahn-Banach theorem, $\scrX_\tau^*$ is not empty --- picking any $\Lambda \in \scrX_\tau^*\subseteq\scrX^*$, there exists some closed interval $I$ such that $C_{\Lambda,I}\supseteq \bbB$, so we can form the intersection
		\begin{equation}
			\tilde{\bbB} = \bigcap_{\substack{ \Lambda\in \scrX_\tau^* ,I \subseteq[-\infty,+\infty] \\ C_{\Lambda,I} \supseteq \bbB }} C_{\Lambda,I}. \label{eq:misc_654}
		\end{equation}
		This is a $\tau$-weakly closed set containing $\bbB$. If $x\notin \bbB$, we can apply the Hahn-Banach separation theorem \cite[Thm. 7.8.6]{TVS} to the sets $\{x\}$ and $\bbB$ to get some $\Lambda \in \scrX_\tau^*$ such that $\Re \Lambda x > 1$ and $\Re \Lambda x_0<1$ for all $x_0\in \bbB$. Then, since $\bbB$ is closed under multiplication by $-1$, $\Re \Lambda x_0 \in (-1,+1)$ for all $x_0\in \bbB$, which means that $C_{\Lambda,[-1,+1]}$ appears on the right-hand side of \cref{eq:misc_654}. 
		
		Since $x\notin C_{\Lambda,[-1,+1]}$, we get $x\notin \tilde{\bbB}$. We conclude that $\tilde{\bbB}=\bbB$ and, therefore, that the latter is $\tau$-weakly closed. 
		\item If $\scrX$ is not separable, then $\tau$ is at least as strong as the weak topology. Since the weak topology of the weak topology is just the weak topology \cite[\S3.10, \S3.11]{Rudin} -- that is, $\sigma(\scrX,\scrX_{\mathrm{w}}^*) = \sigma(\scrX,\scrX^*)$, where $\scrX_{\mathrm{w}} = \sigma(\scrX,\scrX^*)$ -- the $\tau$-weak topology is at least as strong as the weak topology.
	\end{enumerate}
	Thus, the $\tau$-weak topology is admissible. 
\end{proof}

\begin{lemma}
	If $\scrX$ is separable, 
	there exists a countable norming subset $\calS\subseteq \scrX_\tau^*$.
	\label{lem:countable_norming}
\end{lemma}
\begin{proof}
	Let $\{x_n\}_{n=0}^\infty$ denote a dense subset of $\scrX\backslash \{0\}$. By \cite[Thm. 7.8.6]{TVS}, there exists for each $n\in \bbN$ and each $R \in (0,\lVert x_n \rVert )$ an element $\Lambda_{n,R} \in \scrX_\tau^*$ such that $\Re \Lambda_{n,R} x_n >1 $ and $\Re \Lambda_{n,R}<1$ on the closed ball $ R\bbB$ (which is $\tau$-closed by admissibility). Since $R\bbB$ is closed under multiplication by phases, 
	\begin{equation}
		\lVert \Lambda_{n,R} x \rVert < 1
	\end{equation}
	for all $x\in R \bbB$. Thus, $\lVert \Lambda_{n,R} \rVert_{\scrX^*} \leq 1/R$. It follows that $1< \Re \Lambda_{n,R} x_n < | \Lambda_{n,R} x_n| \leq \lVert x_n \rVert/R$, so $\lim_{R\uparrow \lVert x_n \rVert} |\Lambda_{n,R}x_n| = 1$. 
	
	Now let $\calS$ be the set of all functionals of the form $R\Lambda_{n,R}$ for $R$ of the form $\lVert x_n \rVert-1/m$ for $m\in \bbN^+$ sufficiently large such that $1/m < \lVert x_n \rVert$. Then, it is straightforward to check that $\calS$ is a norming subset, and $\calS$ is countable. 
\end{proof}
Cf.\ \cite[Lemma 6.7]{Carothers}. 

\begin{lemma}
	If $\scrX$ is separable and $\calS\subseteq \scrX_\tau^*$ is a norming subset, then the $\sigma(\scrX,\calS)$-topology is admissible. 
	\label{lem:S_admissible}
\end{lemma}
\begin{proof}
	We can assume without loss of generality that, if $\bbK=\bbC$, $e^{i\theta} \Lambda \in \calS$ whenever $\Lambda\in \calS$ and $\theta \in \bbR$. 
	By \cite[Thm.\ 3.10]{Rudin}, the $\sigma(\scrX,\calS)$-topology is an LCTVS topology, and it is no stronger than the norm topology. Consider 
	\begin{equation}
		\tilde{\bbB} = \bigcap_{\substack{ \Lambda\in \calS ,I \subseteq[-\infty,+\infty] \\ C_{\Lambda,I} \supseteq \bbB }} C_{\Lambda,I}, 
		\label{eq:misc_657}
	\end{equation}
	which is a $\sigma(\scrX,\calS)$-closed set containing $\bbB$. If $x\notin \bbB$, then there exists some $\Lambda\in \calS$ such that $|\Re \Lambda x | \in (1,\lVert x_n \rVert]$. Since $\calS$ is norming, $\lVert \Lambda \rVert_{\scrX^*} \leq 1$, so  $C_{\Lambda,[-1,+1]}$ appears on the right-hand side of \cref{eq:misc_657}. But, 
	\begin{equation} 
		x\notin C_{\Lambda,[-1,+1]},
	\end{equation} 
	so $x\notin \tilde{\bbB}$. 
	
	We conclude that $\tilde{\bbB}=\bbB$, so $\bbB$ is $\sigma(\scrX,\calS)$-closed. 
\end{proof}

\printbibliography
 
\end{document}